\documentclass{amsart}[12pt]

\usepackage{amssymb,amsthm}
\usepackage{eucal}
\usepackage[shortlabels]{enumitem}
\usepackage{thmtools,thm-restate}
\usepackage{hyperref}
\usepackage{esvect}
\usepackage{tikz-cd}

\usetikzlibrary{decorations.markings}
\tikzset{negated/.style={
        decoration={markings,
            mark= at position 0.5 with {
                \node[transform shape] (tempnode) {$\backslash$};
            }
        },
        postaction={decorate}
    }
}

\usetikzlibrary{decorations.markings}
\tikzset{boldnegated/.style={
        decoration={markings,
            mark= at position 0.5 with {
                \node[transform shape] (tempnode) {$\textbf{\textbackslash}$};
            }
        },
        postaction={decorate}
    }
}

\setlist[description]{font=\normalfont}

\usepackage{interval}
\intervalconfig{soft  open  fences}
\usepackage{mathtools}

\declaretheorem{theorem}
\declaretheorem[sibling=theorem]{proposition}
\declaretheorem[sibling=theorem]{lemma}

\declaretheorem[sibling=theorem]{claim}

\declaretheorem[style=definition,sibling=theorem]{definition}

\declaretheorem[sibling=theorem]{fact}

\declaretheorem{question}

\DeclareMathOperator{\cf}{cf}

\DeclareMathOperator{\dom}{dom}

\DeclareMathOperator{\range}{range}

\DeclareMathOperator{\stem}{stem}

\DeclareMathOperator{\ssucc}{succ}
\DeclareMathOperator{\osucc}{osucc}

\DeclareMathOperator{\fsplit}{split}

\renewcommand{\th}{^\text{th}}
\newcommand{\su}{\hspace{1mm}\text{s.t.}\hspace{1mm}}

\newcommand{\seq}[2]{\langle #1 : #2 \rangle}

\newcommand{\Col}{\textup{Col}}

\newcommand{\CH}{\textup{\textsf{CH}}}
\newcommand{\ZFC}{\textup{\textsf{ZFC}}}

\newcommand{\LIP}{\textup{\textsf{LIP}}}

\newcommand{\then}{\implies}

\newcommand{\rest}{\upharpoonright}

\newcommand{\nrest}{\!\rest\!}

\renewcommand{\P}{\mathbb P}

\newcommand{\Pcnf}{\mathbb{P}_{\textup{\textsf{CNF}}}}

\newcommand{\Ptanf}{\mathbb{P}_{\textup{\textsf{TANF}}}}

\author{Maxwell Levine}

\title{On Namba Forcing and Minimal Collapses}

\begin{document}

\begin{abstract} We build on a 1990 paper of Bukovsk{\'y} and Copl{\'a}kov{\'a}-Hartov{\'a}. First, we remove the hypothesis of $\CH$ from one of their minimality results. Then, using a measurable cardinal, we show that there is a $|\aleph_2^V|=\aleph_1$-minimal extension that is not a $|\aleph_3^V|=\aleph_1$-extension, answering the first of their questions.\end{abstract}

\maketitle

\section{Introduction}

Forcing extensions often involve collapsing cardinals, so it is natural to study the properties of possible forcing collapses from a general perspective. Namba forcing, which was discovered independently by Bukovsk{\'y} and Namba \cite{Namba1971,Bukovsky1976}, has two notable properties that distinguish it from other forcings that collapse cardinals. First, it allows one to avoid collapsing cardinals below the target: the classical versions singularize $\aleph_2$ while preserving $\aleph_1$. There are limitations, however: Shelah proved that Namba forcing collapses $\aleph_3$ \cite[Theorem 4.73]{Handbook-Eisworth}. The other notable property of Namba forcing is that, like many other tree forcings, it is to some extent minimal. Bukovsk{\'y} and Copl{\'a}kov{\'a}-Hartov{\'a} conducted a thorough investigation into the minimality properties of Namba-like forcings \cite{Bukovsky-Coplakova1990} and their paper became a go-to reference for further work on Namba forcing \cite{Foreman-Todorcevic2005,Guzman-Hrusak-Zapletal2021}.

Bukovsk{\'y} and Copl{\'a}kov{\'a}-Hartov{\'a}'s paper is partially known for a question about collapsing $\aleph_{\omega+1}$ (see \cite{Cummings1997}), but the authors also raised questions about how minimality and control over collapsed cardinals interact. In their notation, a forcing is $|\lambda|=\kappa$ minimal if it forces $|\lambda|=\kappa$ and has no subforcings collapsing $\lambda$ to have cardinality $\kappa$. They showed that $\CH$ implies the $|\aleph_2^V|=\aleph_1$-minimality of classical Namba forcing. We will remove the assumption of $\CH$. They also asked whether there is a $|\aleph_2^V|=\aleph_1$-minimal extension that is not a $|\aleph_3^V|=\aleph_1$-extension \cite[Question 1]{Bukovsky-Coplakova1990}. Assuming the consistency of a measurable cardinal, we will answer their question positively here. Essentially, we are showing that there is more flexibility for producing various $|\lambda|=\kappa$-extensions than was previously known.

The forcing used to obtain these results is a variant of classical Namba forcing in the sense that it is a tree forcing, although in our case the trees will have a height equal to $\omega_1$. An aspect of the main technical idea is present in the recent result that classical Namba forcing consistently has the weak $\omega_1$-approximation property \cite{Levine2023b}. The crux is a sweeping argument that is used to pair the successors of splitting nodes with distinct forced values for a given forcing name. The difference with the present results is that we will use a version of local precipitousness due to Laver to define the splitting behavior of our forcing. This will allow us to use the sweeping argument while ensuring that our forcing is countably closed. We should expect something like this because of the above-mentioned result of Shelah, which in fact shows that \emph{any} extension singularizing $\aleph_2$ to have cofinality $\omega$ while preserving $\aleph_1$ will collapse $\aleph_3$.

The use of large cardinals appears necessary. First, the assumption we employ to define the forcing, which we refer to as the Laver Ideal Property, implies the consistency of a measurable cardinal \cite{Jech-Magidor-Mitchell-Prikry1980}. An extension fitting Bukovsk{\'y} and Copl{\'a}kov{\'a}-Hartov{\'a}'s minimality criteria would likely be a tree forcing, and joint work with Mildenberger \cite{Levine-Mildenberger2024} shows that tree forcings with uncountable height exhibit a number of (rather interesting) pathologies, particularly when it comes to fusion arguments, unless some regularity of their splitting behavior is enforced. Hence, it seems like we need the Laver Ideal Property as long as we expect to use tree forcings. Of course, this does not prove that the large cardinals are necessary. An argument to this effect would probably use an almost disjoint sequence that arises from the failure of a large cardinal principle, and it would need to use the notion of a \emph{strictly} intermediate extension of an arbitrary extension.

\subsection{Definitions and Notation} We assume that the reader is familiar with the basics of set theory  and tree forcings, in particular fusion arguments (see \cite[Chapters 15 and 28]{Jech2003}). Here we will clarify our notation.


\begin{definition} Let $T$ be a tree.
\begin{enumerate}

\item  
For an ordinal $\alpha$, the set
$T(\alpha)$ is the set of $t \in T$ with $\dom(t) =\alpha$.

\item  
The \emph{height} $\operatorname{ht}(T)$ of a tree $T$ is $\min\{\alpha : T(\alpha) = \emptyset\}$.

\item 

We let $[T] = \{f \colon \operatorname{ht}(T) \to \kappa : \forall \alpha < \operatorname{ht}(T), f \rest \alpha \in T\}$. 
Elements of $[T]$ are called \emph{cofinal branches}.
\item 
For
$t_1$, $t_2 \in N \cup [N]$ we write $t_1 \sqsubseteq t_2$ if $t_2 \rest \dom(t_1) = t_1$. The tree order is the relation $\sqsubseteq$. If $t = s \cup \{(\dom(s), \beta)\}$, we write $t = s {}^\frown \langle \beta \rangle$.

\item 
$T \nrest \alpha = \bigcup_{\beta<\alpha}T(\beta)$.

\item 
$T \rest t = \{s \in T : s \sqsubseteq t  \vee t \sqsubseteq s\}$.

\item 
For $t \in T(\alpha)$ we let $\ssucc_T(t) = \{c : c \in T(\alpha+1) \wedge c \sqsupseteq t\}$ denote the \emph{set of immediate successors of $t$}, and $\osucc_T(t) = \{\beta : t {}^\frown \langle \beta \rangle  \in T(\alpha+1)\}$ denote the \emph{ordinal successor set of $t$}.

\item We call $t \in T$ a \emph{splitting node} if $|\ssucc_T(t)|>1$.


\item $\stem(T)$ is the $\sqsubseteq$-minimal splitting node. 

\end{enumerate}
\end{definition}

\begin{definition} Let $\P$ be a tree forcing, loosely defined, with $\mu,\lambda$ fixed as above.

\begin{enumerate}

\item Take $p \in \P$. We let $\fsplit(p)$ denote the set of splitting nodes of $p$. For $\alpha \in \lambda$, $\fsplit_\alpha(p)$ is the set of $\alpha$-order splitting nodes of $p$. 


\item Let $p, q \in \P$, $\alpha < \lambda$. We write $q \leq_\alpha p$ if $q \leq p$, $\fsplit_\alpha(p) = \fsplit_\alpha(q)$, and $\ssucc_p(t)=\ssucc_q(t)$ for all $t \in \fsplit_\alpha(p)$.

\item 
A sequence $\seq{p_\alpha}{\alpha < \delta}$ such that $\delta \leq \lambda$ and for $\alpha <\gamma < \delta$, $p_{\gamma} \leq_\alpha p_\alpha$ is called a \emph{fusion sequence}.
\end{enumerate}
\end{definition}

\section{Results}

\subsection{Minimality without the Continuum Hypothesis} We will discuss the version of Namba forcing that appears in Bukovsk{\'y}'s treatment \cite{Bukovsky1976} since this is the one that appears in Jech's textbook \cite{Jech2003}, which we define here so that there is no risk of ambiguity:

\begin{definition}
The conditions in \emph{classical Namba forcing}, which we denote $\P=\P_{\textup{\textsf{CNF}}}$, consists of conditions that are subsets of ${}^{<\omega} \aleph_2$ such that:

\begin{enumerate}

\item $t \in p$ and $s \sqsubseteq t$ (i.e$.$ s an initial segment of $t$) imply $s \in p$;

\item for all $t \in p$, $|\{\alpha < \aleph_2:t {}^\frown \langle \alpha \rangle \in p\}| \in \{1,\aleph_2\}$;

\item and for all $t \in p$ there is some $s \sqsupseteq t$ such that $s \ne t$ and $|\{\alpha < \aleph_2:t {}^\frown \langle \alpha \rangle \in p\}| = \aleph_2$.

\end{enumerate}

If $p,q \in \P$, then $p \le q$ if and only if $p \subseteq q$.\end{definition}

Bukovsk{\'y} and Copl{\'a}kov{\'a}-Hartov{\'a} showed that $\CH$ implies that $\Pcnf$ is $|\aleph_2^V|=\aleph_1$-minimal \cite[Corollary 1.3]{Bukovsky-Coplakova1990}, and also proved a more general statement, but we will show that the hypothesis of $\CH$ can be dropped if we just want the minimality result for $\Pcnf$. Note that \autoref{no-ch-min} below was specifically proved by Bukovsk{\'y} \cite[Theorems 2,3]{Bukovsky1976} with the assumption of $\CH$. The argument anticipates the main result, \autoref{question1}.

\begin{lemma}\label{no-ch-min} If $G$ is $\Pcnf$-generic over $V$ and suppose $f \in V[G]$ is an unbounded function $\omega \to \theta$ where $\cf^V(\theta) \ge \aleph_2$. Then $V[f]=V[G]$.\end{lemma}

\begin{proof} 
Suppose that $\dot{f}$ is a $\Pcnf$-name for an unbounded function $\omega \to \theta$ and this is forced by some $p \in \Pcnf$. We will define a fusion sequence $\seq{p_n}{n<\omega}$ together with an assignment $\{(t,n_t): t \in \bigcup_{n<\omega}\fsplit_n(p_n)\}$ such that:

\begin{enumerate}
\item for all $t \in \bigcup_{n<\omega}\fsplit_n(p_n)$, $n_t \in \omega$,
\item for each $n<\omega$, $t \sqsubset t'$ implies $n_t < n_{t'}$,
\item for each $n<\omega$ and $t \in \fsplit(p_n)$, there is a sequence $\seq{\gamma_\alpha^t}{\alpha \in \osucc_{p_n}(t)}$ such that $p_n \rest t {}^\frown \langle \alpha \rangle \Vdash \textup{``}\dot{f}(n_t) = \gamma_\alpha^t \textup{''}$ and such that $\alpha \ne \beta$ implies $\gamma_\alpha^t \ne \gamma_\beta^t$. 
\end{enumerate} 

If we define such a sequence and $\bar{p}=\bigcap_{n<\omega}p_n$, then $\bar{p} \Vdash \textup{``}V[\Gamma(\Pcnf)] = V[\dot{f}]\textup{''}$ (where $\Gamma(\Pcnf)$ is the canonical name for the generic), i.e. $\bar{p}$ forces that the generic can be recovered from the evaluation of $\dot{f}$ as in standard minimality arguments. More precisely, we use the fact that the generic is defined by the cofinal branch that it adds. The branch can be defined as the downwards closure of a sequence of splitting nodes $\seq{t_n}{n<\omega}$ that is defined by induction: Let $t_0$ be the $\sqsubseteq$-minimal splitting node. If $t_n$ is defined then $t_{n+1}$ is defined as the next splitting node above $t_{n+1} {}^\frown \langle \alpha \rangle$ where $p \rest t {}^\frown \langle \alpha \rangle$ forces the correct value for $\dot{f} \rest n_{t_k}$.

Formally let $p_{n-1} = p$ and suppose we have defined $p_n$. Define $p_{n+1}$ as follows: For each $t \in \fsplit_n(p_n)$, suppose $s \sqsubset t$ is such that $s \in \fsplit_{n-1}(p_n)$ if $n>0$, and if $n=0$ set $n_s = 0$. We will define a subset $s_t=\seq{\alpha^t_\xi}{\xi<\aleph_2} \subseteq \osucc_{p_n}(t)$, a set of extensions $\seq{q_\xi}{\xi<\aleph_2}$, a set of ordinals $\seq{\gamma_\xi}{\xi < \aleph_2}\subseteq \aleph_2$, and a sequence of natural numbers $\seq{n_\xi}{\xi<\aleph_2}$ such that:

\begin{enumerate}[(i)]
\item $q_\xi \le p_n \rest (t {}^\frown \langle \alpha_\xi \rangle)$,
\item $q_\xi \Vdash \textup{``} \dot{f}(n_\xi) = \gamma_\xi\textup{''}$,
\item $\xi \ne \zeta$ implies $\gamma_\xi \ne \gamma_\zeta$.
\end{enumerate}

We define $\alpha_\xi^t$'s, the $q_\xi^t$'s, and the $\gamma_\xi^t$'s by induction together with a sequence of natural numbers $\seq{m_\xi}{\xi<\aleph_2}$. Suppose we have defined them for $\xi<\zeta<\aleph_2$. We claim that there is some $\beta \in \osucc_{p_n}(t) \setminus \seq{\alpha_\xi^t}{\xi<\zeta}$, some $r \le p_n \rest (t {}^\frown \langle \beta \rangle)$, some ordinal $\delta$, and some $m >n_s$ such that $\gamma \notin \seq{\gamma_\xi^t}{\xi<\zeta }$ and such that $r \Vdash \textup{``}\dot{f}(m) = \delta \textup{''}$. Otherwise it is the case that
\[
\bigcup \{p \rest (t {}^\frown \langle \alpha \rangle): \alpha \in \osucc_{p_n}(t) \setminus \sup_{\xi<\zeta}\alpha_\xi \} \Vdash \textup{``}\dot{f} \rest (n_s,\omega) \subseteq  \seq{\gamma_\xi}{\xi<\zeta} \textup{''},
\]
which contradicts the fact that $p$ forces $\dot{f}$ to be unbounded in $\theta$ where $\theta$ has a cofinality strictly greater than $\aleph_1^V$. Hence we can let $\alpha_\zeta := \beta$, $\gamma_\zeta := \delta$, and $n_\zeta := m$.

Now that the $q_\xi$'s, $\gamma_\xi$'s, and $n_\xi$'s have been defined, let $k<\omega$ be such that there is an unbounded $X \subseteq \aleph_2$ such that $n_\xi = k$ for all $\xi \in X$. Then let $n_t = k$ and let $q_t = \bigcup_{\xi \in X}q_\xi^t$. Finally, let $p_{n+1} = \bigcup\{q_t:t \in \fsplit_n(p_n)\}$. Now that we have defined $\seq{p_n}{n<\omega}$, let $\bar{p} = \bigcap_{n<\omega}p_n$. As argued above, $\bar{p} \Vdash \textup{``}V[\Gamma(\Pcnf)] = V[\dot{f}]\textup{''}$.\end{proof}

\begin{theorem} $\ZFC$ proves that $\P_\textup{\textsf{CNF}}$ is $|\aleph_2^V|=\aleph_1$-minimal.\end{theorem}

\begin{proof} Let $\kappa = \aleph_2^V$ and $\lambda = \aleph_1^V$. Suppose that $G$ is $\Pcnf$-generic over $V$ and that $V \subseteq W \subseteq V[G]$ where $W \models \textup{``}|\kappa| = \aleph_1 \textup{''}$. Consider the case that $\cf^W(\kappa) = \lambda$ as witnessed by some increasing and cofinal $g:\lambda \to \kappa$ in $W$. If $f':\omega \to \kappa$ is the cofinal function added by $\Pcnf$, then in $V[G]$ one can define a cofinal function $h: \omega \to \lambda$ be setting $h(n)$ to be the least $\xi$ such that $f'(n)<g(\xi)$. Then $h$ is cofinal because if $\xi<\lambda$ and $n$ is such that $g(\xi)<f'(n)$, then $h(n)>\xi$. But this implies that $V[G] \models \textup{``}|\lambda|=\omega \textup{''}$, contradicting the fact that $\Pcnf$ preserves $\omega_1$.\footnote{Specifically, $\Pcnf$ and many other variants of Namba forcing in which the trees have height $\omega$ have the property that they preserve stationary subsets of $\omega_1$. A careful and detailed proof for one variant appears in Krueger \cite{Krueger2013}.} Therefore it must be the case that $\cf^W(\kappa) = \omega$ as witnessed by some cofinal $f \in W$, so by \autoref{no-ch-min} we have that $V[f] \subseteq W \subseteq V[G] = V[f]$, hence $W=V[G]$.\end{proof}

\subsection{Developing a Version of Higher Namba Forcing} We will use a notion of Laver to define the forcing.

\begin{definition}[Laver] (See \cite[Chapter X,Definition 4.10]{PIF}.) Given a regular cardinal $\mu$, we write $\LIP(\mu,\lambda)$ if there is a $\mu$-complete ideal $I \subset P(\mu)$ such that there is a set $D \subseteq I^+$ such that:

\begin{enumerate}
\item $D$ is $\lambda$-closed subset in the sense that if $\seq{A_i}{i<\tau}$ is a $\subseteq$-descending sequence of elements of $D$ with $\tau<\lambda$, then $\bigcap_{i<\tau}A_i \in D$,
\item $D$ is dense in $I^+$, i.e$.$ for all $A \in I^+$, there is some $B \subseteq A$ with $B \in I^+$ such that $B \in D$.
\end{enumerate}\end{definition}

\begin{fact}[Laver]\label{gettingLIP} If $\lambda < \mu$ where $\lambda$ is regular and $\mu$ is measurable, then $\Col(\lambda,<\mu)$ forces $\LIP(\mu,\lambda)$.\end{fact}

Laver's proof of \autoref{gettingLIP} is unpublished, but the argument is similar to the one found by Galvin, Jech, and Magidor for obtaining a certain precipitous ideal on $\aleph_2$ \cite{Galvin-Jech-Magidor1978}. Some additional details appear in Shelah \cite[Chapter X]{PIF}.

Now we will define a ``tall'' augmented version of Namba forcing.



\begin{definition} Assume that $\kappa \le \mu < \lambda$ are regular cardinals. Assume $\LIP(\mu,\lambda)$ holds and that $D$ is the dense set witnessing this. Let $\Ptanf^\kappa(D)$ be subsets $p \subseteq {}^{<\kappa}\mu$ such that:

\begin{enumerate}

\item $p$ is a tree, i.e$.$ if $t \in p$ and $s \sqsubseteq t$ then $s \in p$,

\item if $t$ is a splitting node then $\osucc_p(t) \in D$,


\item for all $t \in p$ and $\gamma<\kappa$, there is some $s \sqsupset t$ such that $\dom s \supseteq \gamma$ and $s$ is a splitting node,

\item for all $\sqsubseteq$-increasing sequences of splitting nodes $\seq{t_i}{i<j} \subset p$ with $j<\kappa$, $t^*:=\bigcup_{i<j}t_i \in p$ and $t^*$ is a splitting node.

\end{enumerate} 

For $p,q \in \Ptanf^\kappa(D)$, let $p \le_{\Ptanf^\kappa(D)} q$ if and only if $p \subseteq q$.

In other words, the conditions in $\Ptanf^\kappa(D)$ are Miller-style perfect trees of height $\lambda$ and with club-wise vertical splitting and horizontal splitting sets in $D$.
\end{definition}

Variants of this definition are found throughout the literature starting with work of Kanamori \cite{Kanamori1980}. We will develop the forcing in this section. Many of its properties generalize those of classical Namba forcing, but \autoref{main-minimality} is a delicate point. For the remainder of this section, let $D$ witness $\LIP(\mu,\lambda)$ with respect to an ideal $I$ and let $\P= \Ptanf^\kappa(D)$.


\begin{proposition} $\P$ is $\kappa$-closed.\end{proposition}

\begin{proof} Let $\tau<\kappa$ and suppose $\seq{p_i}{i<\tau}$ is a descending sequence of conditions in $\P$. It is enough to argue that $p:= \bigcap_{i<\tau}p_i \in \P$. We will do so by induction on the size of the intersection, so let as assume that $\bigcap_{i<j}p_i \in \P$ for all $j<\tau$. Since $\emptyset \in p_i$ for all $i<\tau$, it is enough to show that for any $\gamma<\kappa$, there is some $t \in p$ with $\dom t \supseteq \gamma$, $|\ssucc_p(t)|=\mu$, and $\osucc_p(t) \in I^+$. Using our inductive assumption, we can find a sequence $\seq{t_i}{i<\tau}$ such that $t_i \in \bigcap_{j<i}p_j$ is a splitting node for all $i<\tau$. Then let $t^* = \bigcup_{i<\tau}t_i$. For all $i<\tau$, $t^*$ is a union splitting nodes of $p_i$, so it is a splitting node of $p_i$. Let $O_i=\osucc_{p_i}(t^*)$. Then $\bigcap_{i<\lambda}O_i \in I^+$ by the closure property of $\LIP(\mu,\lambda)$, so we are done. (See \cite{Kanamori1980}.) \end{proof}


\begin{proposition}\label{main-singularization} $\Vdash_\P \textup{``}\cf(\mu)=\kappa \textup{''}$.\end{proposition}

\begin{proof} By a density argument, the intersection of all conditions in a $\P$-generic filter is a $\kappa$-sequence consisting of ordinals in $\mu$. Let $\dot{f}$ be the name for the function sending the $i\th$ point of the sequence to its corresponding ordinal.\end{proof}

This is our main lemma. The crux is the sweeping argument in \autoref{basic-claim}.

\begin{lemma}\label{main-minimality} $\P$ is $(\cf(\mu)=\kappa)$-minimal.\end{lemma}

\begin{proof} Suppose that $\dot{f}$ is a $\P$-name forced by the empty condition to be a cofinal function $\kappa \to \mu$.

We define the main idea of the proof presently. Let $\varphi(q,i)$ denote the formula
\begin{align*}
 i < \kappa \ \wedge \ & q \in \P \wedge \exists \seq{a_\alpha}{\alpha \in \osucc_q(\stem(q))} \su \\
&\forall \alpha \in \osucc_q(\stem(q)),q \rest (\stem(q) {}^\frown \langle \alpha \rangle) \Vdash `\dot{f} \rest i = a_\alpha\textup{'} \wedge \\
& \forall \alpha, \beta \in \osucc_q(\stem(q)), \alpha \ne \beta \then a_\alpha \ne a_\beta.
\end{align*}

\begin{claim}\label{basic-claim} $\forall j<\kappa, p  \in \P,\exists i \in (j,\kappa),q \le p \su \stem(p)=\stem(q) \wedge \varphi(q,i)$.\end{claim}

\begin{proof}
First we establish a slightly weaker claim: for all splitting nodes $t \in p$, there is a sequence $\seq{(q_\alpha,i_\alpha,a_\alpha)}{\alpha \in \osucc_p(t)}$ such that:

\begin{enumerate}[(i)]
\item $\forall \alpha \in \osucc_p(t)$, $\forall i_\alpha \in (j,\kappa)$,
 $q_\alpha \le p \rest (t {}^\frown \langle \alpha \rangle)$, and
 $q \rest (t {}^\frown \langle \alpha \rangle) \Vdash ``\dot{f} \rest i_\alpha = a_\alpha\textup{''}$,
\item $\alpha \ne \beta \then a_\alpha \ne a_\beta$.
\end{enumerate}

We define this sequence by induction on $\alpha \in \osucc_p(t)$. Suppose we have $\seq{(q_\beta,i_\beta,a_\beta)}{\beta \in \alpha \cap \osucc_p(t)}$ such that (i) and (ii) hold below $\alpha$. Then we can argue that there is a triple $(r,i,a)$ such that $r \le p \rest (t {}^\frown \langle \alpha \rangle)$ and $r \Vdash \textup{``}\dot{f} \rest i = a \text{''}$ and $a \notin \{a_\beta:\beta \in \alpha \cap \osucc_p(t)\}$. If not, this means that
\[
p \rest (t {}^\frown \langle \alpha \rangle)  \Vdash \textup{``}\bigcup_{i<\lambda}\dot{f} \rest i \subseteq \{a_\beta:\beta \in \alpha \cap \osucc_p(t)\} \textup{''}.
\]
(Note that by $\lambda$-closure, $\P$ forces $\textup{``}\dot{f} \rest i \in V\textup{''}$ for all $i<\kappa$.) This is a contradiction because $\bigcup_{\beta \in \alpha \cap \osucc_p(t)}a_\beta$ has cardinality less than $\mu$ and $\dot{f}$ is forced to be unbounded in $\mu$. Since the triple that we want exists, we can let $(q_\alpha,i_\alpha,a_\alpha)$ be such a triple.

Now that we have established the slightly weaker claim, apply the $\mu$-completeness of $I$ to find some $S' \subseteq \osucc_p(t)$ such that $S' \in I^+$ and there is some $i$ such that $i_\alpha = i$ for all $\alpha \in S'$. Then choose $S \subseteq S'$ such that $S \in D$ using the density property indicated by $\LIP(\mu,\lambda)$ and let $q = \bigcup_{\alpha \in S}q_\alpha$.\end{proof}

Now that we have our claim, we can use it to construct a fusion sequence $\seq{p_\xi}{\xi<\lambda}$ such that:

\begin{enumerate}[(i)]
\item $\forall \xi<\kappa,\forall \zeta<\xi,t \in \fsplit_\zeta(p_\xi),\exists k<\kappa,\varphi(p_\zeta \rest t,k)$;
\item $\forall \xi<\kappa,t \in \fsplit_\xi(p_\xi),\forall \zeta<\xi$, if $k_\zeta$ is such that $\varphi(p_\zeta \rest ,k_\zeta)$ holds, then $k_\xi$ can be chosen so that $k_\xi>\sup_{\zeta<\xi}k_\zeta$.
\end{enumerate}

We define the sequence by cases: Let $p_0$ be arbitary. If $\xi<\kappa$ is a limit then we let $p_\xi = \bigcap_{\zeta<\xi}p_j$. Now suppose that $\xi=\xi'+1$. Then for all $t \in \fsplit_{\xi'}(p_{\xi'})$, apply \autoref{basic-claim} to obtain some $q \le p_{\xi'} \rest t$ with $\stem q = t$ such that for some $k_\xi>\sup_{\zeta \le \xi'}k_\zeta$, $\varphi(q,k_\xi)$ holds. Finally, having defined the fusion sequence, we let $p=\bigcap_{\xi<\lambda}p_\xi$.

Now we argue that $p \Vdash \textup{``}\Gamma(\P) \in V[\dot{f}]\textup{''}$. Let $f=\dot{f}[G]$ for some $G$ that is $\P$-generic over $V$. We will argue that $G$ is definable from $f$. Specifically, we will define an $\le$-increasing sequence $\seq{i_\xi}{\xi<\kappa} \subseteq \kappa$ and a $\sqsubseteq$-increasing sequence $\seq{t_\xi}{\xi<\kappa} \subset p$ of splitting nodes such that for all $\xi<\kappa$:

\begin{enumerate}[(a)]
\item $t_\xi \in \fsplit_\xi(p_\xi)$,
\item there is some $q \in G$ such that $t_\xi \in q$.
\end{enumerate}
 Then it will be the case that $G = \{q \in \P:\bigcup_{\xi<\lambda}t_\xi \subseteq q\}$, i.e$.$ the $t_\xi$'s define the generic branch.

Let us construct the sequence. We can let $t_0= \emptyset$. If $\xi$ is a limit then we let $t_\xi = \bigcup_{\zeta<\xi}t_\zeta$, and we note that $t_\xi$ is a splitting node of the correct order by the definition of the poset. If $\xi = {\xi'}+1$ then we consider $t_{\xi'}$, which by assumption is an element of $\fsplit_{\xi'} p_{\xi'}$ such that $\varphi(p_{\xi'} \rest t_{\xi'},i)$ holds for some $i$. Choose $\tilde{t} \in \ssucc_{t_{\xi'}}p_{\xi'}$ determining the correct values for $\dot{f}[G]$ up to $i$.\end{proof}

\begin{proposition}\label{main-highpres} $\P$ does not add surjections from $\kappa$ to $\theta$ for any regular $\theta>\mu$.\end{proposition}

\begin{proof} Suppose that we have a $\P$-name $\dot{f}$ for a function such that (without loss of generality) the empty condition forces $\dot{f}:\kappa \to \theta$. We will define a fusion sequence $\seq{p_i}{i<\kappa}$ as follows: Let $p_0$ be arbitrary. If $i$ is a limit then let $p_i = \bigcap_{j<i}p_j$. If $i=k+1$, then for all $t \in \fsplit_{k}(p_{k})$ and $\alpha \in \osucc_{p_k}(\alpha)$, choose some $q_{t,\alpha} \le p_k \rest (t {}^\frown \langle \alpha \rangle)$ deciding $\dot{f}(k)$. Then let $p_i = \bigcup \{q_{t,\alpha}:t \in \fsplit_{k}(p_{k}), \alpha \in \osucc_{p_k}(t)\}$.

If we let
\[
B=\{\delta:\exists i<\lambda,t \in \fsplit_i(p_i),\alpha \in \osucc_{p_i}(t),q_{t,\alpha} \Vdash \textup{``}\dot{f}(i)=\delta\textup{''}\},
\]
then it follows that $p \Vdash \textup{``} \range(\dot{f}) \subseteq \sup(B)<\theta \textup{''}$.\end{proof}

Now we are in a position to answer the question of Bukovsk{\'y} and Copl{\'a}kov{\'a}-Hartov{\'a} that was mentioned in the introduction.

\begin{theorem}\label{question1} Assuming consistency of a measurable cardinal, there is a model $V$ such that there is an $|\aleph_2^V|=\aleph_1$-minimal extension $W \supset V$ that is not an $|\aleph_3^V|=\aleph_1$-extension.\end{theorem}

\begin{proof} Suppose that $\LIP(\aleph_2,\aleph_1)$ holds and is witnessed by the dense set $D_2$ in $V$. Then let $W$ be an extension by $\Ptanf^{\aleph_1}(D_2)$. Then $W$ is an $|\aleph_2^V|=\aleph_1$-extension by \autoref{main-singularization} and it is a minimal such extension by \autoref{main-minimality}. If $\Ptanf^{\aleph_1}(D_2)$ collapses $\aleph_3^V$, then it would collapse it to an ordinal of cardinality $\aleph_1^V$, but it follows from \autoref{main-highpres} that this is not possible.\end{proof}

\subsection{Remaining Questions}

As stated in the introduction, it would be clarifying to know for sure whether there is an exact equiconsistency.

\begin{question} Does the conclusion of \autoref{question1} require consistency of a measurable cardinal?\end{question}

There is also the question of the extent to which \autoref{question1} can be stratified. 

\begin{question} Assuming $\LIP(\lambda,\mu)$, is it consistent that $\omega<\kappa < \lambda < \mu$ are regular cardinals and $\Ptanf^\kappa$ preserves cardinals $\nu \le \lambda$?\end{question}

This question appears to rely heavily on the determinacy of the generalizations of Namba-style games to uncountable length $\kappa$ (see e.g$.$ \cite[Chapter XI]{PIF} \cite{Foreman-Todorcevic2005} , \cite[Fact 5]{Cummings-Magidor2011}). One could pose this question in terms of $(\kappa,\nu)$-distributivity, but even some tricks that allow one to merely obtain cardinal preservation from similar posets (see \cite[Theorem 3]{Levine2023b}) seem to depend on these types of games (see \cite[Fact 1]{Cummings-Magidor2011}).

\bibliographystyle{plain}
\bibliography{bibliography}

\end{document}